\documentclass[12pt]{article}

\usepackage{fullpage}
\usepackage{amsfonts,amssymb,stmaryrd,amsthm,amsmath}
\usepackage{frcursive}

\usepackage{xcolor}
\usepackage{graphicx,tikz,tabularx}
\usetikzlibrary{matrix,arrows,decorations.markings}

\newcommand{\mb}[1]{{#1}}

\newcommand\longto{{\longrightarrow}}

\newcommand\CC{\mathbb{C}}
\newcommand\NN{\mathbb{N}}\newcommand\ZZ{\mathbb{Z}}
\DeclareMathOperator{\Spec}{Spec}
\newcommand\bAA{\mathbb{A}}

\newcommand{\bl}{\ell}
\newcommand\inv{{^{-1}}}

\newcommand\Aut{{\operatorname{Aut}}}
\newcommand\End{{\operatorname{End}}}
\newcommand\Mc{{\mathcal M}}
\newcommand\Ad{{\operatorname{Ad}}}\newcommand\Hom{{\operatorname{Hom}}}
\newcommand\GL{{\operatorname{GL}}}\newcommand\ad{\operatorname{ad}}
\newcommand\Id{{\operatorname{Id}}}
\newcommand\Ker{{\operatorname{Ker}}}

\renewcommand\sl{{\mathfrak{sl}}}\newcommand\gl{{\mathfrak{gl}}}

\renewcommand{\lg}{{\mathfrak g}}
\newcommand{\lh}{{\mathfrak h}}\newcommand{\lz}{{\mathfrak z}}\newcommand{\lc}{{\mathfrak c}}
\newcommand{\lb}{{\mathfrak b}}
\newcommand{\oIb}{{\overline{I\lb}}} 
\newcommand{\lw}{{\mathfrak w}}
\renewcommand{\ln}{{\mathfrak n}}
\newcommand{\lp}{{\mathfrak p}}
\renewcommand{\lq}{{\mathfrak q}}
\newcommand{\la}{{\mathfrak a}}
\newcommand{\bla}{\bar{\la}}
\newcommand{\lgt}{\tilde{\mathfrak g}}
\newcommand{\lbt}{\tilde{\mathfrak b}}
\newcommand{\lnt}{\tilde{\mathfrak n}}
\newcommand{\oIlb}{\overline{I\lb}}
\newcommand{\lgkm}{{\mathfrak g}^{\rm KM}}
\newcommand{\lbkm}{{\mathfrak b}^{\rm KM}}
\newcommand{\lnkm}{{\mathfrak n}^{\rm KM}}

\newcommand{\bolda}{\boldsymbol{\lambda}}
\newcommand{\boldmu}{\boldsymbol{\mu}}
\newcommand{\func}[4]{
\begin{array}{c c l}#1&\longrightarrow & #2\\ #3&\longmapsto &
                                                               #4\end{array}
}

\newcommand{\Dt}{\tilde{\mathcal D}}
\newcommand{\mcM}{\mathcal M}
\newcommand{\vV}{{\mathcal V}}

\usepackage[utf8]{inputenc}  

\usepackage[T1]{fontenc}

\usepackage{hyperref}

\newtheorem{prop}{Proposition}
\newtheorem{theo}[prop]{Theorem}
\newtheorem{lemma}[prop]{Lemma}
\newtheorem{coro}[prop]{Corollary}

\newenvironment{remark}{\noindent{\bf Remark.}}{}

\newcommand\DynkinNodeSize{2mm}
\newcommand\DynkinArrowLength{3mm}
\tikzset{
  dnode/.style={
    circle,
    inner sep=0pt,
    minimum size=\DynkinNodeSize,
    fill=white,
    draw},
  middlearrow/.style={
    decoration={markings,
      mark=at position 0.6 with
      {\draw (0:0mm) -- +(+135:\DynkinArrowLength); \draw (0:0mm) -- +(-135:\DynkinArrowLength);},
    },
    postaction={decorate}
  },
  leftrightarrow/.style={
    decoration={markings,
      mark=at position 0.999 with
      {
      \draw (0:0mm) -- +(+135:\DynkinArrowLength); \draw (0:0mm) -- +(-135:\DynkinArrowLength);
      },
      mark=at position 0.001 with
      {
      \draw (0:0mm) -- +(+45:\DynkinArrowLength); \draw (0:0mm) -- +(-45:\DynkinArrowLength);
      },
    },
    postaction={decorate}
  },
  sedge/.style={
  },
  dedge/.style={
    middlearrow,
    double distance=0.5mm,
  },
  tedge/.style={
    middlearrow,
    double distance=1.0mm+\pgflinewidth,
    postaction={draw}, 
  },
  infedge/.style={
    leftrightarrow,
    double distance=0.5mm,
  }
}

\begin{document}
\title{\bf On the automorphisms of the Drinfel'd double of a Borel Lie subalgebra}
\author{Michaël Bulois and Nicolas Ressayre}

\maketitle

\begin{abstract}
Let $\lg$ be a complex simple Lie algebra with a Borel subalgebra $\lb$. 
Consider the semidirect product   $I\lb=\lb\ltimes\lb^*$, where 
the dual $\lb^*$ of $\lb$ is equipped with the coadjoint action of
$\lb$ and is considered as an abelian ideal of  $I\lb$. 
We describe the automorphism group $\Aut(I\lb)$ of the Lie algebra $I\lb$.
In particular we prove that it contains the automorphism group of  the extended Dynkin diagram of $\lg$.
In type $A_n$, the dihedral subgroup  was recently proved to be contained 
in $\Aut(I\lb)$ by Dror
Bar-Natan and Roland van der Veen in \cite{BNVDV} (where $I\lb$ is
denoted by $I{\mathfrak u}_n$). 
Their construction is handmade and they asked for an explanation
which is provided by this note.
Let $\ln$ denote the nilpotent radical of $\lb$. We obtain
  similar results
for $I\lb=\lb\ltimes\ln^*$ that is both an Inönü-Wigner contraction of
$\lg$ and the quotient of $I\lb$ by its center.
\end{abstract}

\section{Introduction}

Given any complex Lie algebra $\la$, one can form the {\it associated
  Drinfeld double} $(I\la,\la)$. Here, 
$I\la=\la\ltimes\la^\ast$ is  the semidirect product of $\la$ with its dual
$\la^\ast$ where $\la^\ast$ is considered as an abelian ideal
and $\la$ acts on $\la^\ast$ via the coadjoint action. 

As mentioned in \cite{BNVDV}, for applications in knot theory and
representation theory, the most important case is when $\la=\lb$ is the
Borel subalgebra of some simple Lie algebra $\lg$. It is precisely the
situation studied here.
In addition to \cite{BNVDV}, several examples of these algebras appear
with variations in the literature. In \cite{NW}, Nappi-Wittney use the
case when $\lg=\sl_2$ in conformal field theory. Several authors also
consider $\oIb:=\lb\ltimes \ln^*$ where $\ln$ is the derived subalgebra of
$\lb$. It is the quotient of $I\lb$ by its center.
Note that $\lb\ltimes \ln^*$ is a contraction of $\lg$ (see
Section~\ref{sec:defIb} for details).
In \cite{KZ:Brauer}, Knutson and Zinn-Justin meet this algebra for $\lg=\mathfrak{gl}_n$ in the associative setting, see below. 
In \cite{F1, F2}, Feigin uses $\lb\ltimes \ln^*$ in order to study degenerate flag varieties for $\lg=\mathfrak{sl}_n$. 
For a general semisimple Lie algebra $\lg$, in \cite{PY1}, Panyushev and Yakimova study the invariants of $\lb\ltimes \ln^*$  under the action of their adjoint group.
Finally, in \cite{PY2,Ph}, similar considerations are studied replacing $\lb$ by an arbitrary parabolic subalgebra of $\lg$.

The aim of this note is to give  new interpretations of $I\lb$ and
$\oIb$ in the language of Kac-Moody algebras and to completely
describe the automorphism groups of $I\lb$ and $\oIb$. 

Before describing this group, we introduce some notation.
Let $r$ denote the rank of $\lg$ and $G$ the adjoint group with Lie
algebra $\lg$. Let $B$ be the Borel subgroup of $G$ with $\lb$ as Lie
algebra.
Consider two abelian additive groups: the quotient $\lg/\lb$ and the
space $\Mc_r(\CC)$ of $r\times r$-matrices.

An important ingredient is the extended Dynkin diagram of $\lg$. 
On Figure~\ref{fig:extD}, these diagrams and
their automorphism groups are shortly recalled (see Section~\ref{sec:affLA}). The notation $D_{(\ell)}$ stands for the dihedral group of order $2\ell$, not to be confused with the Dynkin diagram of type $D_{\ell}$.

The following is the main result of the paper.
\begin{theo}
\label{th:AutIntro}
The neutral component $\Aut(I\lb)^\circ$ of the automorphism group $\Aut(I\lb)$ of the
Lie algebra $I\lb$ decomposes as
$$
\CC^*\ltimes\bigg(
(B\ltimes\lg/\lb)\times \Mc_r(\CC)
\bigg ).
$$
The group of components $\Aut(I\lb)/\Aut(I\lb)^\circ$ is isomorphic to
the automorphism group of the extended Dynkin diagram of $\lg$ and can
be lifted to a subgroup of $\Aut(I\lb)$.
\end{theo}

The details of how these subgroups act on $I\lb$ are given in
Section~\ref{sec:subgroups}. Section~\ref{sec:structAut} explain how
the semidirect products are formed. 

\begin{figure}
  \centering
  \begin{center}
  \begin{tabular}{|c|c|c|}
\hline
  \begin{tikzpicture}[scale=.7,baseline=(2.center)]
    \draw (0.5,1.4) node[anchor=north]  {$\tilde{A}_1$};

    \node[dnode, fill=black] (1) at (0,0) {};
    \node[dnode] (2) at (1,0) {};

    \path (2) edge[infedge] (1)
          ;
\end{tikzpicture}&
\begin{tikzpicture}[scale=.7,baseline=(2.center)]
\draw (2.5,2.4) node[anchor=north]  {$\tilde{A}_\bl\;(\bl\geq 2)$};
    \node[dnode] (1) at (0,0) {};
    \node[dnode] (2) at (1,0) {};
    \node[dnode] (3) at (2,0) {};
    \node[dnode] (4) at (3,0) {};
    \node[dnode] (5) at (4,0) {};
    \node[dnode] (6) at (5,0) {};
    \node[dnode,fill=black] (0) at (2.5,1) {};
    \path (1) edge[sedge,dashed] (2) edge[sedge] (0);
    \path (3) edge[sedge] (2) edge[sedge] (4);
     \path (6) edge[sedge] (5) edge[sedge] (0);
\path (4) edge[sedge] (5);
  \end{tikzpicture}
  &
  
  \begin{tikzpicture}[scale=.7,baseline=(3.center)]
\draw (2.5,2) node[anchor=north]  {$\tilde{B}_\bl\;(\bl\geq 3)$};
    \node[dnode] (1) at (0,0.7) {};
    \node[dnode] (2) at (1,0) {};
    \node[dnode] (3) at (2,0) {};
    \node[dnode] (4) at (3,0) {};
    \node[dnode] (5) at (4,0) {};
    \node[dnode] (6) at (5,0) {};
    \node[dnode,fill=black] (0) at (0,-0.7) {};
    \path (2) edge[sedge,dashed] (3) edge[sedge] (0) edge[sedge] (1);
    \path (4) edge[sedge] (3) edge[sedge] (5);
     \path (5) edge[dedge] (6);
  \end{tikzpicture}\\[4.5ex]

$\Aut(\Dt)=\ZZ/2\ZZ$&$\Aut(\Dt)=D_{(\bl+1)}$&$\Aut(\Dt)=\ZZ/2\ZZ$\\
\hline
\begin{tikzpicture} [scale=.7,baseline=(2.center)]
\draw (1,1.4) node[anchor=north]  {$\tilde{G}_2$};
    \node[dnode,fill=black] (1) at (0,0) {};
    \node[dnode] (2) at (1,0) {};
    \node[dnode] (3) at (2,0) {};
 
     \path (2) edge[tedge] (3);
\path (1) edge[sedge] (2);
  \end{tikzpicture}
   &

\begin{tikzpicture}[scale=.7,baseline=(2.center)]
    \draw (2.5,1.4) node[anchor=north]  {$\tilde{C}_\bl\;(\bl\geq 2)$};
\node[dnode,fill=black] (1) at (0,0) {};
    \node[dnode] (2) at (1,0) {};
    \node[dnode] (3) at (2,0) {};
    \node[dnode] (4) at (3,0) {};
    \node[dnode] (5) at (4,0) {};
    \node[dnode] (6) at (5,0) {};
    \path (1) edge[dedge] (2);
    \path (3) edge[sedge,dashed] (2) edge[sedge] (4);
     \path (6) edge[dedge] (5);
\path (4) edge[sedge] (5);
  \end{tikzpicture}
&
 \begin{tikzpicture}[scale=.7,baseline=(3.center)]
  \draw (2.5,2) node[anchor=north]  {$\tilde{D}_\bl\;(\bl\geq 5)$};
  \node[dnode] (1) at (0,0.7) {};
    \node[dnode] (2) at (1,0) {};
    \node[dnode] (3) at (2,0) {};
    \node[dnode] (4) at (3,0) {};
    \node[dnode] (5) at (4,0) {};
    \node[dnode] (6) at (5,0.7) {};
     \node[dnode] (7) at (5,-0.7) {};
    \node[dnode,fill=black] (0) at (0,-0.7) {};
    \path (2) edge[sedge,dashed] (3) edge[sedge] (0) edge[sedge] (1);
    \path (4) edge[sedge] (3) edge[sedge] (5);
     \path (5) edge[sedge] (6) edge[sedge] (7);
  \end{tikzpicture}
\\[3.6ex]
$\Aut(\Dt)$ is trivial&$\Aut(\Dt)=\ZZ/2\ZZ$&$\Aut(\Dt)=D_{(4)}$\\
\hline
\begin{tikzpicture}[scale=.7,baseline=(1.center)]
\draw (2,1.1) node[anchor=north]  {$\tilde{E}_6$};
    \node[dnode,fill=black] (0) at (0,0) {};
    \node[dnode] (1) at (1,0) {};
    \node[dnode] (2) at (2,0) {};
\node[dnode] (3) at (3,-0.7) {};
    \node[dnode] (4) at (3,0.7) {};
    \node[dnode] (5) at (4,-0.7) {};
    \node[dnode] (6) at (4,0.7) {};

    \path (1) edge[sedge] (0) edge[sedge] (2);
    \path (2) edge[sedge] (3) edge[sedge] (4);
     \path (5) edge[sedge] (3);\path (6) edge[sedge] (4);
  \end{tikzpicture}
&
\begin{tikzpicture} [scale=.7,baseline=(4.center)]
\draw (4,1.4) node[anchor=north]  {$\tilde{E}_7$};
    \node[dnode,fill=black] (1) at (0,0) {};
    \node[dnode] (2) at (1,0) {};
    \node[dnode] (3) at (2,0) {};
    \node[dnode] (4) at (3,0) {};
    \node[dnode] (5) at (4,0) {};
    \node[dnode] (6) at (5,0) {};
    \node[dnode] (7) at (6,0) {};
    \node[dnode] (0) at (3,1) {};
    \path (4) edge[sedge] (0) edge[sedge] (5) edge[sedge] (3);
    \path (2) edge[sedge] (1) edge[sedge] (3);
     \path (6) edge[sedge] (5) edge[sedge] (7);
  \end{tikzpicture}
&
\begin{tikzpicture} [scale=.7,baseline=(3.center)]
\draw (2.5,1.4) node[anchor=north]  {$\tilde{F}_4$};
    \node[dnode,fill=black] (1) at (0,0) {};
    \node[dnode] (2) at (1,0) {};
    \node[dnode] (3) at (2,0) {};
    \node[dnode] (4) at (3,0) {};
    \node[dnode] (5) at (4,0) {};
    
    \path (2) edge[sedge] (1) edge[sedge] (3);
   
     \path (3) edge[dedge] (4);
\path (4) edge[sedge] (5);
  \end{tikzpicture}\\[3ex]
$\Aut(\Dt)={\mathfrak S}_3$&$\Aut(\Dt)=\ZZ/2\ZZ$&$\Aut(\Dt)$ is trivial\\
\hline
 \begin{tikzpicture}[scale=.7,baseline=(3.center)]
  \draw (1,2) node[anchor=north]  {$\tilde{D}_4$};
  \node[dnode] (1) at (0,0.7) {};
    \node[dnode] (2) at (1,0) {};
    \node[dnode] (6) at (2,0.7) {};
     \node[dnode] (7) at (2,-0.7) {};
    \node[dnode,fill=black] (0) at (0,-0.7) {};
    \path (2)  edge[sedge] (0) edge[sedge] (1) edge[sedge] (6) edge[sedge] (7);
  \end{tikzpicture}
&
\multicolumn{2}{|c|}{
\begin{tikzpicture} [scale=.7]
\draw (2.5,1.4) node[anchor=north]  {$\tilde{E}_8$};
    \node[dnode,fill=black] (0) at (-1,0) {};
    \node[dnode] (1) at (0,0) {};
    \node[dnode] (2) at (1,0) {};
    \node[dnode] (3) at (2,0) {};
    \node[dnode] (4) at (3,0) {};
    \node[dnode] (5) at (4,0) {};
    \node[dnode] (6) at (5,0) {};
    \node[dnode] (7) at (6,0) {};
    \node[dnode] (8) at (4,1) {};
    \path (5) edge[sedge] (6) edge[sedge] (4) edge[sedge] (8);
    \path (3) edge[sedge] (4) edge[sedge] (2);
     \path (1) edge[sedge] (0) edge[sedge] (2);
     \path (6) edge[sedge] (7);
  \end{tikzpicture}}\\
$\Aut(\Dt)=\mathfrak S_4$  &\multicolumn{2}{|c|}{$\Aut(\Dt)$ is trivial}\\
\hline
\end{tabular}
  \end{center}
  \caption{Extended Dynkin diagrams and their automorphisms}
  \label{fig:extD}
\end{figure}

One of the amazing facts is that the extended Dynkin diagram of $\lg$
plays a crucial role in $\Aut(I\lb)$. 
On one hand, we explain this by constructing
the extended Cartan matrix of $\lg$ in terms of $I\lb$ in
Section~\ref{sec:roots}. 
On the other hand, this diagram is the Dynkin diagram
of the untwisted affine Lie algebra constructed from the loop algebra of $\lg$.
A second explanation is given by Theorem~\ref{th:main} that realizes $I\lb$
as a subquotient of the affine Lie algebra associated to $\lg$. 

More generally, $I\lb$ is a degeneration $\lim_{\epsilon\rightarrow0} \lg_+^{\epsilon}$ with  $\lg_+^{\epsilon}\cong\lg\oplus\lh$ for $\epsilon \in \CC\setminus \{0\}$. In Section \ref{sec2}, we explain how to interpret this degeneration in the affine Lie algebra setting. We also study the possible lifting of $\theta\in \Aut(\Dt)$ to $\Aut(\lg^{\epsilon}_+)$, see Section \ref{sec:AutDtilde}.

\bigskip
{\bf Link with other works.}
In \cite{KZ:Brauer}, Knutson and Zinn-Justin defined a degeneration
$\bullet$ of the standard associative product on $\mcM_n(\CC)$.
Let $\lb$ denote the set of upper triangular matrices. Identifying
the vector space $\mcM_n(\CC)$ with $\lb\times \mcM_n(\CC)/\lb$ in a natural way one gets 
$$
(R,L)\bullet (V,M)=(RV,RM+LV),
$$
for any $R,V\in\lb$ and $L,M\in \mcM_n(\CC)/\lb$.
The Lie algebra of the group $(\mcM_n(\CC),\bullet)^\times$ of
invertible elements of this algebra is $\lb\ltimes \mcM_n(\CC)/\lb$,
where the product is defined similarly to that of $I\lb$.
Note also that a cyclic automorphism (corresponding in our setting
to the cyclic automorphism of the extended Dynkin diagram of type
$A_{n-1}$ and with the unexpected cyclic automorphism of \cite{BNVDV})
appears in \cite{KZ:Brauer}. Moreover \cite[Proposition~2]{KZ:Brauer},
which realizes $(\mcM_n(\CC),\bullet)$ as a subquotient of $\mcM_n(\CC[t])$,
is similar to our Theorem~\ref{th:main}.

A generalization of $\overline{I\lb}$ is the following: fix a simple Lie algebra
$\lg$  and a parabolic subalgebra $\lp$ of $\lg$. Let $\ln_{\lp}^-(\cong \lg/\lp)$ be the nilradical of a parabolic subalgebra of $\lg$ opposite to $\lp$. Then $\lq_{\lp}:=\lp\ltimes \ln_{\lp}^-$ is also a degeneration of $\lg$. In the study of semi-invariants of $\lq_{\lp}$ some data linked with the extended Dynkin diagram also make appearance in \cite[Theorem~5.5]{Ya} (Borel case) and in \cite[Proposition~5.2.1]{Ph} (general case). In type $A_{n-1}$, standard parabolics are characterized by an ordered partition $\bolda=(\lambda_1,\dots,\lambda_k)$ of $n$. Transforming $\bolda$ into $\boldmu:=(\lambda_k,\lambda_1,\dots,\lambda_{k-1})$, the cyclic action of $\ZZ/n\ZZ$ coming from the symmetries of the extended Dynkin diagrams described in \cite{BNVDV} allows to write $\lq_{\lp_{\bolda}}\cong \lq_{\lp_{\boldmu}}$. 
This explains many symmetries noted in \cite{Ph}, see (3.9) in \textit{loc. cit.}.

\bigskip
{\bf Motivation and story of the present work.}
In \cite{BNVDV}, Bar-Natan and van der Veen constructed an ``unexpected'' cyclic
automorphism of $I\lb$ when $\lg=\gl_n(\CC)$. The first version of
this work was an explanation of this automorphism by using affine Lie
algebras. Simultaneously with this first version, A.~Knutson mentioned
to Bar-Natan his earlier work \cite{KZ:Brauer} with Zinn-Justin.

\bigskip
{\bf Acknowledgements.} We are very grateful to Dror Bar Natan for
useful discussions that had motivated this work.
The authors are partially supported by the French National Agency
(Project GeoLie ANR-15-CE40-0012).

\section{The Lie algebras $I\lb$, $\lg^\epsilon_+$ and $\lg\otimes
  \CC[t^{\pm 1}]$}

\label{sec2}

\subsection{Definitions of $I\lb$ and $\lg^\epsilon_+$}

\label{sec:defIb}

Let $\lg$ be a complex simple Lie algebra with Lie bracket denoted by
$[\ ,\ ]$. 
Fix a Borel subalgebra $\lb$ of $\lg$ and a Cartan subalgebra
$\lh\subset\lb$. 
Let $\lb^-$ be the 
Borel subalgebra of $\lg$ containing $\lh$ which is opposite to $\lb$.
Set $\vV=\lb\oplus\lb^-$ viewed as a vector space. 
In this section, we define the Lie bracket $[\ ,\ ]_\epsilon$ on $\vV$
depending on the complex parameter $\epsilon$, interpolating
between $I\lb$ and  the direct product $\lg\oplus\lh$.

Let $\ln$ and $\ln^-$ denote the derived subalgebras of $\lb$ and
$\lb^-$ respectively. 
Fix $\epsilon\in \CC$. 
Define the skew-symmetric bilinear 
bracket $[\ ,\ ]_\epsilon$ on $\vV$ by

$$
\begin{array}{l@{\,}lll}
[x,x']_\epsilon  &=[x,x']&\forall x,x'\in\lb\\[0pt] 
[y,y']_\epsilon  &=\epsilon [y,y']&\forall y,y'\in\lb^-\\[6pt] 
[x,y]_\epsilon&=(\epsilon X+ \epsilon \frac H 2,\frac H 2+Y)&\forall
                                          x\in\lb\;y\in\lb^-&\mbox{where }[x,y]=X+H+Y\in\ln\oplus\lh\oplus\ln^-
\end{array}
$$

Then $[\ ,\ ]_\epsilon$ satisfies the Jacobi identity (see discussion
after \eqref{eq:g1product} for a proof).
Endowed with this Lie bracket, $\vV$ is denoted by
$\lg^\epsilon_+$. 
The linear
  map

$$
\begin{array}{cccl@{\hspace{1cm}}l}
  \varphi_\epsilon\,:&\lb\oplus\lb^-&\longto&\lb\oplus\lb^-\\
&(x,y)&\longmapsto&(x,\epsilon y)&{\rm for\ any\ }x\in\lb,\,y\in\lb^-
\end{array}
$$
allows to interpret $\lg^\epsilon_+$ as an Inönü-Wigner contraction
\cite{IW} of
$\lg^1_+$. Indeed, for any nonzero $\epsilon$, we have 
\begin{equation}
  \label{eq:4}
  [X,Y]_\epsilon=\varphi_\epsilon\inv([\varphi_\epsilon(X),
  \varphi_\epsilon(Y)]_1)\qquad\forall X,Y\in\vV.
\end{equation}

We now describe  $\lg^1_+$. Using the triangular decomposition
\begin{equation}
  \label{eq:td}
  \lg=\ln\oplus\lh\oplus\ln^-,
\end{equation}
one defines the injective linear map
$$
\begin{array}{cccl}
  \iota^1_\lg\,:&\lg=\ln\oplus\lh\oplus\ln^-&\longto&\lg^1_+\\
&(\xi,\alpha,\zeta)&\longmapsto&(\xi+\frac \alpha 2,\frac \alpha 2  +\zeta)
\end{array}
$$  
and checks that it is a Lie algebra homomorphism whose image is an ideal of $\lg^1_+$. 
Moreover, the image of
$$
\begin{array}{cccl}
  \iota^1_\lh\,:&\lh&\longto&\lg^1_+\\
&\alpha&\longmapsto&(-\alpha, \alpha)
\end{array}
$$
is the center of $\lg^1_+$ and, as Lie algebras,
\begin{equation}
  \label{eq:g1product}
  \lg^1_+=\iota_\lg^1(\lg)\oplus\iota_\lh^1(\lh).
\end{equation}

Observe that we never used the Jacobi identity for $[\ ,\ ]_1$ to prove
the isomorphism~\eqref{eq:g1product}. Hence, we can deduce from it that
$[\ ,\ ]_1$ satisfies the Jacobi identity. Then, the expression~\eqref{eq:4} 
implies that $[\ ,\ ]_\epsilon$ satisfies the Jacobi identity
for any nonzero $\epsilon$. Since this property is closed on the space of bilinear maps,
it is satisfied by $[\ ,\ ]_0$ too.

\bigskip
Consider now $I\lb$ with its Lie bracket $[\;,\;]_{I\lb}$  defined as follows:
$\lb^*$ is an abelian ideal on which $\lb$ acts by the coadjoint action.
Denote by
$\kappa\,:\,\lg\,\longto\,\lg^*$ the Killing form on $\lg$.
Since the orthogonal complement of $\lb$ with respect to $\kappa$ is $\ln$, $\lb^*$
identifies with $\lg/\ln$ as a $\lb$-module. 
Identify $\lg/\ln$ with $\lb^-$ in a canonical way
(that is by $y\in\lb^-\longmapsto y+\ln$) and denote by
$\pi\,:\,\lg\,\longto\, \lb^-$ the quotient map.
Then $Ib=\lb\oplus\lb^*$ identifies with $\lb\oplus\lb^-=\vV$.
Let $[\;,\;]_I$ denote the Lie bracket transferred to  $\vV$ from $[\;,\;]_{I\lb}$.  
 Let 
$x,x'\in \lb$ and $y,y'\in \lb^-$ and decompose $[x,y']-[x',y]$ as
$X+H+Y$ with respect to $\lg=\ln\oplus\lh\oplus\ln^-$. Then
 
\begin{equation}
  \label{eq:2I}
  [(x,y),(x',y')]_I=([x,x'],H+Y).
\end{equation}

\bigskip
We now describe  $\lg^0_+$.
The Lie bracket $[\ ,\ ]_0$ on $\vV=\lg^0_+$ is given by
\begin{equation}
  \label{eq:20}
  [(x,y), (x',y')]_0=([x,x'],\frac H 2+Y).
\end{equation}
Comparing~\eqref{eq:2I} and \eqref{eq:20}, one gets that the following
linear map $\eta$ is a Lie algebra isomorphism between $\lg^0_+$ and $I\lb$:
$$
\begin{array}{cccc}
  \eta\,:&\vV=\lb\oplus(\lh\oplus\ln^-)&\longto&\lb\oplus\lb^*=I\lb\\
&(x,h,y)&\longmapsto&(x,\kappa(2 h +y,\square)).
\end{array}
$$

  Replacing $\lb^-$ and $\lb^*$ by $\ln^-$ and $\ln^*$ respectively,
  one defines $\lg^\epsilon$ and one gets the isomorphsims
  $\lg\simeq\lg^\epsilon$ (for any $\epsilon\neq 0$) and $\lg^0\simeq\oIb$.

\subsection{The affine Kac-Moody Lie algebra}
\label{sec:affLA}


The untwisted affine Kac-Moody Lie
algebra $\lgkm$ is constructed from  the simple Lie algebra $\lg$. We refer to \cite[Chapters I and XIII]{Ku} for the basic properties of $\lgkm$. 
Denote by $\lz(\lgkm)$ the one dimensional center of $\lgkm$.
Consider the  Borel subalgebra $\lbkm$ of $\lgkm$ and its derived subalgebra $\lnkm$.
By killing the semi-direct product and the central extension from the
construction of $\lgkm$, one gets
$$
\begin{array}{rcl}
  \lgt&:=&[\lgkm,\lgkm]/\lz(\lgkm)\\
&\cong&\CC[t^{\pm 1}] \otimes \lg,
\end{array}
$$
and
$$
\begin{array}{rcl}
  \lbt&:=&(\lbkm\cap [\lgkm,\lgkm])/\lz(\lgkm)\subset \lgt\\
\lnt&:=&(\lnkm\cap [\lgkm,\lgkm])/\lz(\lgkm)=[\lbt,\lbt].
\end{array}
$$
Identify $\lg$ with the subspace $\CC\otimes\lg\subset \lgt$. 
Note that $\lgkm/\lz(\lgkm)=\tilde\lg+\CC d$ where $d$ acts as the derivation with respect to $t$.

We consider the set of (positive) roots $\Phi^{(+)}$ (resp. $\tilde \Phi^{(+)}$)  of $\lg$ (resp. $\lgkm$) and the set of simple roots $\Delta$ (resp. $\tilde\Delta$) with respect to $\lh\subset \lb\subset \lg$ (resp. $\lh+\CC d+\lz(\lgkm)\subset \lbkm\subset\lgkm$).
%
We recall the following classical facts: 
\[\ln^{KM}\cong \lnt=\bigoplus_{\alpha\in \tilde \Phi^+} \lgt_{\alpha}\]
where $\lgt_{\alpha}\cong \lg^{KM}_{\alpha}$ is the root space
associated to $\alpha$. Moreover, $\lnt$ is generated, as a Lie algebra by the subspaces $(\lgt_{\alpha})_{\alpha\in \tilde \Delta}$.
The identification of $\Delta$ with $\{\alpha\in \tilde \Delta\,| \,\alpha(d)=0\}$ yields the above-described embedding $\lg\subset \lgt$. 
Denoting by $\delta$ the indivisible positive imaginary root in $\tilde \Phi$, we have
\[\tilde \Phi=\{n\delta+\alpha\,|\, \alpha\in \Phi\cup\{0\}, n\in \ZZ\}\setminus \{0\}\]
\[\tilde \Delta=\Delta\cup\{\alpha_0+\delta\}\]
where $\alpha_0$ is the lowest root of $\Phi$. Note that $\lgt_{n\delta}=t^n\lh$ ($n\in \ZZ$), using the notation $\lgt_{0}:=\lh$. 

Finally, the extended Dynkin diagram can be reconstructed from the combinatorics of $\tilde\Delta$ in $\tilde \Phi$. Indeed, the nodes correspond to the elements of $\tilde \Delta$ and the non-diagonal entries $a_{\alpha,\beta}$ of the generalized Cartan matrix (encoding the arrows of the diagram) are $a_{\alpha,\beta}=-\max \{n\in \NN| \beta+n\alpha \in \tilde \Phi\}$ by Serre relations.

We list in Figure~\ref{fig:extD} the extended Dynkin diagram
$\tilde{\mathcal D}_{\lg}$ in each simple type. The black node
corresponds to the simple root $\alpha_0+\delta$. We also provide the
automorphism group of $\tilde{\mathcal D}_{\lg}$. Note that by the
definition of $\lgkm$ given in \cite[\S1.1]{Ku}, any
$\theta\in\Aut(\tilde{\mathcal D}_{\lg})$ provides an automorphism
$\theta^{KM}\in \Aut(\lgkm)$ stabilizing both $\lh+\CC d+\lz(\lgkm)$ and $\lbkm$ \mb{and permuting the generators $e_{\alpha}, f_{\alpha}$ ($\alpha\in \tilde\Delta$) via $\theta^{KM}(e_{\alpha})=e_{\theta(\alpha)}$ and $\theta^{KM}( f_{\alpha}) =f_{\theta(\alpha)}$}. Since $\lz(\lgkm)$ and $[\lgkm,\lgkm]$ are characteristic in $\lgkm$, \emph{i.e.} stabilized by any automorphism of Lie algebra, this yields an automorphism $\tilde \theta\in \Aut(\lgt)$. \mb{Note that some choices have to be made for $\theta^{KM}(d)$, but the automorphism $\tilde \theta$ only depends on $\theta$ and on the $e_{\alpha}$, $f_{\alpha}$ ($\alpha\in \tilde\Delta$), since those elements generate $\lgt$.}

In a first version of this paper, it was claimed that $\tilde\theta$ is $\CC[t]$-linear. In fact, it is unclear whether this result holds in general. However, we have the following
\begin{lemma} \label{Ct-lin} \mb{Under above notation, there exists $\lambda\in \{\pm1\}$ such that
\[\forall x\in \lgt, \, \tilde \theta(tx)=\lambda t\tilde \theta(x).\]
In particular, the automorphism $\tilde\theta\in \Aut(\lgt)$ stabilizes $t\lnt$.\\
Moreover, $\lambda=1$ whenever the order of $\theta$ is odd.} 
\end{lemma}
\begin{proof}
Note that, since $\tilde{\theta}$ comes from an element $\theta^{\rm KM}\in \Aut(\lgkm)$, its action on the semi-group $\tilde\Phi^+$ stabilizes the semi-group of positive imaginary roots $\NN^*\delta$ and thus fixes its generator $\delta$. In particular, in the additive group $\tilde\Phi\cup \{0\}$, we have $\tilde \theta(\cdot+\delta)=\delta+\tilde \theta(\cdot)$. Defining $\Psi$ on $\lgt$ via $\Psi(x)=\tilde\theta^{-1}(t^{-1}\tilde \theta(tx))$, we thus get that $\Psi_{\alpha}:=\Psi_{|\lgt_{\alpha}}$ is an invertible linear map on $\lgt_{\alpha}$ for any $\alpha\in \tilde\Phi\cup \{0\}$. Since $\dim \lgt_{\alpha}=1$ for $\alpha\in \tilde \Phi\setminus\ZZ\delta$, we can thus define $\lambda_{\alpha}$ as the element of $\CC^{\times}$  such that $\Psi_{\alpha}=\lambda_{\alpha}Id_{\lgt_{\alpha}}$.

Let $\alpha,\beta\in \tilde \Phi\cup \{0\}$, $x_{\alpha}\in \lgt_{\alpha}$, $x_{\beta}\in \lgt_{\beta}$. By $\CC[t]$-bilinearity of the bracket, we get \begin{equation}\Psi_{\alpha+\beta}([x_{\alpha},x_{\beta}])=\tilde \theta^{-1}(t^{-1}[\tilde \theta(tx_{\alpha}),\tilde \theta (x_{\beta})])=[\Psi_{\alpha}(x_{\alpha}),x_{\beta}].
\label{eq_Psi}\end{equation}

For $\alpha=0$, $x_{\alpha}=h\in \lh$ and $\beta\in \tilde \Phi\setminus\ZZ\delta$, we get \begin{equation}\lambda_{\beta}\beta(h)x_{\beta}=\Psi_{\beta}(\beta(h)x_{\beta})\stackrel{\eqref{eq_Psi}}=\beta(\Psi_0(h))x_{\beta}\label{eq_Psi0}.\end{equation} In particular, $\Psi_0$ induces on $\lh^*$ a linear map ${}^t\Psi_0$ sending $\beta$ to $\lambda_{\beta}\beta$ for each $\beta\in \Phi\subset \tilde \Phi\setminus\ZZ\delta$. If $\beta,\gamma\in \Delta$ correspond to connected nodes of the Dynkin diagram of $\lg$, then $\beta, \gamma$ and $\beta+\gamma$ are eigenvectors of ${}^t\Psi_0$ so $\lambda_{\beta}=\lambda_{\gamma}$. By connexity of the Dynkin diagram, we get that the $\lambda_{\beta}$ ($\beta\in \Delta$) are all equal to a single value $\lambda$. Since $\Delta$ generates $\lh^*$, we get $\Psi_0=\lambda Id_{\lgt_0}$.

For any $\beta\in \tilde\Phi\setminus\ZZ\delta$, we can choose $h\in \lh$ such that $\beta(h)\neq0$. Applying \eqref{eq_Psi0} yields $\lambda_{\beta}\beta(h)x_{\beta}=\beta(\lambda h) x_{\beta}$, that is $\lambda_{\beta}=\lambda$. 

When $\alpha=-\beta\in \Delta$, $n\in \ZZ$, we get $\Psi_{n\delta}(t^n[x_{\alpha},x_{-\alpha}])\stackrel{\eqref{eq_Psi}}=[\Psi_{\alpha}(x_{\alpha}), t^nx_{-\alpha}]=\lambda t^n [x_{\alpha},x_{-\alpha}]$. Since the $t^n[\lgt_{\alpha},\lgt_{-\alpha}]$ ($\alpha\in \Delta$) generate $\lgt_{n\delta}$, this yields $\Psi_{n\delta}=\lambda Id_{\lgt_{n\delta}}$. Finally, we have proved that $\Psi=\lambda Id_{\lgt}$ and this yields the first assertion of the Lemma.

\mb{Let $m$ be the order of $\theta$. The first assertion of the lemma can be rewritten as $t^{-1}\tilde\theta t=\lambda\tilde \theta$ where $t^{\pm1}$ denotes the multiplication by $t^{\pm1}$ in $\lgt$. This identity to the power $m$ yields $\lambda^m=1$.}

In the setting of \cite[Chapter XIII]{Ku}, the Cartan involution $\omega$ of $\lgt$ sending each generator $e_{\alpha}$ ($\alpha\in \tilde \Delta$) to $-f_{\alpha}$ is given by 
\[\omega(t^ix)=t^{-i}\mathring\omega(x)\qquad  (i\in \ZZ, x\in \lg) \]
where $\mathring \omega$ is the Cartan involution of $\lg$. As a consequence, $\omega t=t^{-1}$. Also, $\omega \circ \tilde\theta\circ \omega (e_{\alpha})=\omega\circ  \tilde\theta(-f_{\alpha})=-\omega(f_{\theta(\alpha)})=e_{\theta(\alpha)}=\tilde\theta(e_{\alpha})$ and the same computation gives $\omega \circ \tilde\theta\circ \omega (f_{\alpha})=\tilde\theta(f_{\alpha})$ so $\omega \tilde\theta \omega=\tilde\theta$.
Then conjugating $t^{-1}\tilde\theta t=\lambda\tilde \theta$ by the involution $\omega$ yields $t\tilde\theta t^{-1}=\lambda\tilde \theta$. It follows from these equalities that $\lambda^2=1$.  \mb{Hence $\lambda\in \{\pm1\}$ with $\lambda=1$ if $m$ is odd.}

\mb{Finally, $\tilde\theta$ permutes the generators of $\lnt$: $(e_{\alpha})_{\alpha\in \tilde\Delta}$. Hence $\tilde\theta$ stabilizes $\lnt$ and $\tilde\theta(t\lnt)=\pm t\lnt=t\lnt$}

\end{proof}
\begin{remark} 
We also checked in several cases, including the cyclic automorphism in type A, that $\lambda=1$.
In such cases, $\tilde \theta$ then also stabilizes $(t-\epsilon)\lnt$ for any $\epsilon\in \CC$.
\end{remark}

\subsection{Realization of $\lg^\epsilon_+$}
\label{sec:KM_interpret}

The Lie  algebras $\lbt$ and $\lnt$
 decompose as
$$
\begin{array}{ccl}
  \lbt&=&\CC[t]\lb\oplus t\CC[t]\ln^-,\\
\lnt&=&\CC[t]\ln\oplus t\CC[t]\lb^-.
\end{array}
$$ 
Moreover, $(t-\epsilon)\lnt$ is an ideal of $\lbt$, and
$\lbt/((t-\epsilon)\lnt)$ is a Lie algebra. 

\begin{theo} \label{th:main} 
Let $\epsilon\in\CC$. The Lie algebras $\lg^\epsilon_+$ and
$\lbt/(t-\epsilon)\lnt$ are isomorphic.
Similarly, $\lg^{\epsilon}$ is isomorphic to $\lbt/(t-\epsilon)\lbt$.
\end{theo}

\begin{proof}
From Section~\ref{sec:defIb}, we have 
$
\lg_{+}^{1}= \lb\oplus\lb^{-}$ as vector spaces. Elements of $\lg^1_+$ will be written as couples with respect to this decomposition.

Set $\widetilde{\lg_+^1}:=\CC[t^{\pm1}]\otimes \lg_+^1$ and extend
$\iota^1_{\lg}$ to an injective $\CC[t^{\pm1}]$-linear  map
$\tilde{\lg}\rightarrow \widetilde{\lg_+^1}$. 
Consider the subspace $\lw:=\CC[t]\lb\oplus t\CC[t]\lb^{-}$ that is a
Lie subalgebra of $\widetilde{\lg_+^1}$. 
If $\epsilon \neq 0$, the Inönü-Wigner contraction \eqref{eq:4} on $\lg^1_+$ with respect to the
decomposition $\lb\oplus \lb^{-}$ 
gives rise to $\lg_{+}^{\epsilon}$ ($\epsilon\in \CC$). 
We easily deduce that the linear map
\begin{equation}
  \begin{array}{ccl@{\hspace{0.8cm}}l}
    \lg_{+}^{\epsilon}&\longto&\lw/(t-\epsilon)\lw\\
(x,y)&\longmapsto&x+ty+(t-\epsilon) \lw&\mbox{for any }x\in\lb\mbox{
                                         and }y\in\lb^-,
  \end{array}
\label{IW_iso}\end{equation}
is a Lie algebra isomorphism. For $\epsilon=0$, it is still a linear isomorphism and, by continuity, a Lie algebra homomorphism.

Set $\lb^{-}_0:=\iota^{1}_{\lg}(\lb^{-})=\{(h,h)|h\in \lh\}\oplus
\ln^{-}$. 
Observe that $t\lb^-_0$ is contained in $\lw$. 
Indeed, for any $h\in \lh$, the element $t(h,h)=t(h,0)+t(0,h)$ belongs
to $\CC[t]\lb\oplus t\CC[t]\lb^-$. 
In particular, one gets a linear map induced by the inclusions of
$\lb$ and $t\lb^-_0$ in $\lw$:
$$
\lb\oplus t\lb^-_0\longto\lw.
$$
One can easily check that it induces a linear isomorphism $\lb\oplus t\lb^-_0\longto\lw/(t-\epsilon)\lw$. Setting $\lbt_{\lw}:=\langle \lb\oplus t\lb_0^{-}\rangle_{Lie}$, the Lie subalgebra of $\lw$ generated by $\lb\oplus t\lb_0^{-}$, we thus get a Lie algebra isomorphism.
\begin{equation}
\lbt_{\lw}/((t-\epsilon)\lw\cap \lbt_{\lw})\longto
\lw/(t-\epsilon)\lw.\label{qiso}
\end{equation} 
Since, $\lb=\{(h,0)|h\in \lh\}\oplus \iota^1_{\lg}(\ln)$ and $ \langle \iota^1_{\lg}(\ln)\oplus \iota^1_{\lg}(t\lb^{-})\rangle_{Lie}=\iota^1_{\lg}(\langle \ln\oplus t\lb^{-}\rangle_{Lie})=\iota^1_{\lg}(\lnt)$, 
we have
\begin{equation}\lbt_{\lw}
=\{(h,0)|h\in \lh\}\oplus \iota^1_{\lg}(\lnt)\stackrel{}{\cong}
\iota^{1}_{\lg}(\lbt) \stackrel{}{\cong}\lbt,\label{blw}
\end{equation} 
the middle Lie algebra isomorphism being the identity on $\iota^1_{\lg}(\lnt)$
and sending $(h,0)$ to $\frac 12(h,h)$ for each $h\in \lh$. 
Moreover, $(t-\epsilon)\lw\cap
\lbt_{\lw}=(t-\epsilon)\iota^1_{\lg}(\lnt)$. 
Indeed, 
$(t-\epsilon) \iota^1_{\lg}(\lnt)$ is contained in
$(t-\epsilon)\lw\cap \lbt_{\lw}$, and 
$\lb\oplus t\lb^{-}_0$ is complementary to
$(t-\epsilon)\iota^1_{\lg}(\lnt)$ in $\lbt_{\lw}$.

We finally get the desired Lie isomorphism \[\lbt/(t-\epsilon)\lnt\stackrel{\eqref{blw}}{\cong} \lbt_{\lw}/(t-\epsilon)\iota^1_{\lg}(\lnt)\stackrel{\eqref{qiso}}{\cong} \lw/(t-\epsilon)\lw\stackrel{\eqref{IW_iso}}{\cong}\lg_{+}^{\epsilon}\]
\end{proof}

In addition, we can make explicit the isomorphism of Theorem~\ref{th:main}:
$$\begin{array}{cccll}
  \gamma_{\epsilon}\,:&\lg^\epsilon_+&\stackrel{\cong}{\longto}&\lbt/(t-\epsilon)\lnt\\
&(x,0)&\longmapsto&x&{\rm if\ } x\in\ln\\
&(0,y)&\longmapsto&ty&{\rm if\ } y\in\ln^-\\
&(a,b)&\longmapsto& (a-\epsilon b)+2tb &{\rm if\ } a,b\in\lh\\
\end{array}
$$
and its inverse map is induced by
$$
\begin{array}{cccll}
  \theta\,:&\lbt&\longto&\vV\\
&Px&\longmapsto&P(\epsilon)x&{\rm if\ } x\in\ln\\
&tRy&\longmapsto&R(\epsilon)y&{\rm if\ } y\in\ln^-\\
&Qh&\longmapsto& (\frac{Q(\epsilon)+Q(0)}{2}h,\frac{Q(\epsilon)-Q(0)}{2\epsilon}h)&{\rm if\ } h\in\lh\, (\epsilon\neq 0)\\[0.6em]
& & & (Q(0)h,\frac12Q'(0)h) & {\rm if\ } h\in\lh \, (\epsilon= 0)
\end{array}
$$

Note that, in order to prove Theorem~\ref{th:main}, we could
alternatively have checked directly that $\theta$ is a surjective Lie algebra
homomorphism from $\lbt$ onto $\lg^\epsilon_+$ with kernel $(t-\epsilon)\lnt$.


\section{Some subgroups of $\Aut(I\lb)$}
\label{sec:subgroups}
\subsection{The roots of $I\lb$}
\label{sec:roots}
From Sections~\ref{sec:defIb} and~\ref{sec:KM_interpret}, we can
interpret the Lie algebra $I\lb$ in the Kac-Moody world via the
isomorphism 
\[\func{I\lb}{\lbt/t\lnt}{(x,y)}{x+ty} \qquad \left(\begin{array}{c}x\in\lb,\\ y\in \lb^{-}\cong \lg/\ln\stackrel{\kappa}{\cong}\lb^*\end{array}\right)\]
From now on, this identification will be made systematically. In particular, we write $I\lb=\lb\oplus t\lb^{-}$. 
We first describe some basic properties of $I\lb$ in this language. 
\begin{lemma} 
\label{lem:structure_IB}
  \begin{enumerate}
\item The subalgebra $\lc:=\lh\oplus t\lh$ is a Cartan subalgebra of $I\lb$. 
Namely, $\lc$ is abelian and equal to its normalizer.
\item Under the action of $\lc$, $I\lb$ decomposes as
$$
I\lb=\lc\oplus \bigoplus_{\alpha\in\Phi^+}\lg_\alpha\oplus \bigoplus_{\alpha\in\Phi^-}t\lg_\alpha.
$$
For $\alpha\in\Phi^+$, $\lc$ acts on $\lg_\alpha$ 
with the  weight $(\alpha,0)\in \lh^*\times t\lh^*$. 
For $\alpha\in\Phi^-$, $\lc$ acts on $t\lg_\alpha$ 
with the  weight $(\alpha,0)\in \lh^*\times t\lh^*$. 
Here, we identified $\lc^*$ with $\lh^*\times t\lh^*$ in a natural way. 

\item The set of $\ad$-nilpotent elements of $I\lb$ is $\lnt/t\lnt=\ln\oplus t\lb^{-}$.
  \item The center of $I\lb$ is $\lz(I\lb)=t\lh$.
\item The derived subalgebra of   $I\lb$ is
  $[I\lb;I\lb]=\lnt/t\lnt$.
\end{enumerate}
\end{lemma}

\begin{proof}
1-2)  The fact that $\lc$ is abelian and the decomposition in
$\lh$-eigenspaces are clear from the definition of $\lgt$.  
The action of $t\lh$ is zero since it sends $\lnt$
 to $t\lnt$ that vanishes itself in $I\lb$.
The decomposition of $I\lb$ in weight spaces under the action of $\lc$
follows.
Then this decomposition also implies that $\lc$ is its own normalizer in $I\lb$.\\
3) The elements of $\lnt/t\lnt$ are clearly $\ad$-nilpotent. From 2), an element with nonzero component in $\lh$ is not $\ad$-nilpotent.\\
4) Since $t\lh$ acts as $0$ on $\lnt/t\lnt$ and on $\lh$, we have $t\lh\subset \lz(I\lb)$. The decomposition in weight spaces implies the converse inclusion.\\
5) The inclusion $[I\lb,I\lb]\subset \lnt/t\lnt$ is clear. On the
other hand we deduce from the weight space decomposition that the
subspaces $(\lgt_{\alpha})_{\alpha\in \tilde \Delta}$ belong to $[I\lb,I\lb]$. Since they generate $\lnt$ in $\lgt$, the result follows.
\end{proof}

It follows from Lemma~\ref{lem:structure_IB} and Theorem~\ref{th:main}
that $\oIlb\cong I\lb/t\lh\cong \lg^0_+/\lz(\lg^0_+)\cong \lg^0$. Then
it is straightforward from Lemma~\ref{lem:structure_IB} and its proof that
\begin{itemize}
\item $\lh$ is a Cartan subalgebra of $\oIlb$.
\item The non-zero $\lh$-weights (resp. weight spaces) on $\oIlb$ coincide with the non-zero $\lc$-weights (resp. weight space) on $I\lb$ via projection. In particular $\Phi(\oIlb)\cong\Phi(I\lb)\cong  \Phi$.
\item $[\oIlb,\oIlb]=\lnt/t\lbt$.
\end{itemize}

From Lemma~\ref{lem:structure_IB} (2), the set $\Phi(I\lb)$ of nonzero weights of $\lc$
acting on $I\lb$ identifies with $\Phi$. It is also useful to embed
$\Phi(I\lb)$ in $\tilde\Phi$ by 
$$
\begin{array}{cccl}
  \varphi\,:&\Phi(I\lb)&\longto&\tilde\Phi\\
&\alpha\in\Phi^+&\longmapsto&\alpha\\
&\alpha\in\Phi^-&\longmapsto&\delta+\alpha
\end{array}
$$
Indeed, the weight space $(I\lb)_\alpha$ identifies with
$\lgt_{\varphi(\alpha)}$, for any $\alpha\in \Phi(I\lb)$. In particular, for $\alpha, \beta\in \tilde \Phi\cup \{0\}$, we have $[I\lb_{\varphi^{-1}(\alpha)}, I\lb_{\varphi^{-1}(\beta)}]\subset I\lb_{\varphi^{-1}(\alpha+\beta)}$ with equality when $\alpha,\beta,\alpha+\beta\notin \{0,\delta\}$.
Set also $\Delta(I\lb)=\varphi\inv(\tilde\Delta)=\Delta\cup\{\alpha_0\}$.

\begin{lemma}\label{Ib2}
\begin{enumerate}
\item  The derived subalgebra of  $I\lb^{(1)}:=[I\lb,I\lb]$ is
  $$I\lb^{(2)}=t\lh\oplus\bigoplus_{\alpha\in\Phi(I\lb)\setminus \Delta(I\lb)}(I\lb)_{\alpha}$$

\item Assume that $\lg$ is not $\mathfrak{sl}_2$. 
For $\alpha,\beta\in
   \Delta(I\lb)$ ($\alpha\neq\beta$), the corresponding entry of the
  generalized Cartan Matrix of $\lgkm$ is given by
$$
a_{\alpha,\beta}=-\max \{n\in \NN\;|\; \beta+n\alpha \in  \Phi(I\lb)\}.
$$ 
\end{enumerate} 
\end{lemma}

\begin{proof}
1) Recall that $\lnt$ is generated as a Lie algebra by the
$(\lgt_{\alpha})_{\alpha\in \tilde \Delta}$. Thus, for weight reasons,
the $(\lgt_{\alpha})_{\alpha\in \tilde\Phi\setminus \tilde \Delta}$
are root spaces included in $[\lnt,\lnt]$. Since $\tilde \Delta$ is a
linearly independent set, they are in fact the only root spaces not
contained in $[\lnt,\lnt]$.
Taking a quotient, this yields $\bigoplus_{\alpha\in\Phi(I\lb)\setminus \Delta(I\lb)}(I\lb)_{\alpha}= I\lb^{(2)}$.  

2) Recall that the statement is valid if we replace $\Phi(I\lb)$ by
$\tilde \Phi$, see Section~\ref{sec:affLA}. It is thus sufficient to
show that \[\beta+n\alpha\in \tilde\Phi\ \Rightarrow\ \beta+n\alpha\in \Phi(I\lb).\] 
When $\alpha,\beta\in \Delta$, the statement is clear since $\Phi^{+}\subset \Phi(I\lb)$.\\
If $\beta=\delta+\alpha_0$, then $\beta+n\alpha\in \tilde\Phi$ means that $\alpha_0+n\alpha\in \Phi$. Expressing $\alpha_0$ as a linear combination of simple roots, one gets only negative coefficients. Since $\lg$ is not $\sl_2$,  some of them remain negative in the expression of $\alpha_0+n\alpha$, 
so this root has to lie in $\Phi^{-}$. Thus $\beta+n\alpha\in \Phi(I\lb)$.\\
If $\alpha=\delta+\alpha_0$, then $\beta+n\alpha\in \tilde\Phi$ means that $\beta+n\alpha_0\in \Phi$. 
For height reasons, 
we must have $n\in\{0,1\}$. Then, $\beta+n\alpha\in \Phi(I\lb)$.
\end{proof}

\begin{remark}
  One can observe that the first assertion of Lemma~\ref{Ib2} is
  similar to
$$
[\ln,\ln]=\bigoplus_{\alpha\in\Phi^+\setminus \Delta}\lb_{\alpha}.
$$
\end{remark}

\subsection{The adjoint subgroup of $\Aut(I\lb)$}

Let $G$ be the adjoint group with Lie algebra $\lg$. Let $T$ and $B$
be the connected subgroups of $G$ with Lie algebras $\lh$ and $\lb$.
Consider now $\lb^{-}\cong\lg/\ln$ equipped with the addition as an abelian
algebraic group. 
The adjoint action of $B$ on $\lg$ stabilizes $\ln$ and
induces a linear action on $\lb^{-}\cong\lg/\ln$ by group isomorphisms.
We can construct the semidirect product:
$$
IB:=B\ltimes \lb^{-}.
$$ 
By construction the Lie algebra of $IB$ identifies with $I\lb$. 
The adjoint action of $IB$ on $I\lb$ is given by
\begin{equation}
  \begin{array}{ccl@{\quad}l}
   IB\times I\lb &\longto&I\lb\\
((b,f),x+ty)&\longmapsto&b\cdot x+ tb\cdot(y+[f,x]+\ln)&\mbox{for }b\in
                                                      B,\,x\in\lb\mbox{
                                                      and } f, y\in\lb^-,
  \end{array}
\label{eq:action_ADIB}\end{equation}
where $y+[f,x]+\ln$ is viewed as an element of $\lg/\ln\cong \lb^{-}$ and where $\cdot$ denotes the $B$-action on $\lb$ and on $\lb^{-}$.
It induces a group homomorphism
$$
\Ad\,:\, IB\longto \Aut(I\lb)
$$
with kernel $Z(IB)\cong(1,\lh)$. In particular, one gets: 

\begin{lemma}
  The image $\Ad(IB)$ is isomorphic to $B\ltimes \lg/\lb$.
\end{lemma}

Note also that $\Ad(IB)=H\ltimes (N\ltimes \lg/\lb)$ where $N$ and $H$ are the connected subgroups of $B$ with respective Lie algebras $\ln$ and $\lh$. Since $\ln+t \lb^{-}$ is the set of ad-nilpotent elements of $I\lb$, we get the following result from \eqref{eq:action_ADIB}.
\begin{lemma}
\begin{enumerate}
\item The group of elementary automorphisms $\Aut_e(I\lb)=\exp \ad (\ln+t \lb^{-})$ coincides with $N\ltimes \lg/\lb$.
\item $\Ad(IB)=\exp\ad (I\lb)$
\end{enumerate}
\end{lemma}

\subsection{A unipotent subgroup of $\Aut(I\lb)$}







Let $\la$ be a Lie algebra. We consider the derived subalgebra $\la^{(1)}:=[\la,\la]$, the center $\lz:=\lz(\la)$ and the quotient Lie algebra $\bla:=\la/\lz$.
\newline\indent Any linear map $u\in \Hom(\la/\la^{(1)},\lz)$, defines a linear map $\bar u:\left\{\func{ \la}{\la}X{X+u(X)}\right.$. Since $u$ takes values in $\lz$ and vanishes on $\la^{(1)}$, we have 
\[[\bar u(X),\bar u(Y)]=[X+u(X),Y+u(Y)]=[X,Y]=[X,Y]+u([X,Y])=\bar u([X,Y]).\]
In other words, $\bar u$ is a morphism of Lie algebras.
\newline\indent On the other hand, any $\theta\in \Aut(\la)$
stabilizes the center of $\la$, and hence it induces an automorphism
of $\bla$. This yields a natural group homomorphism
\begin{equation}R:\Aut(\la)\rightarrow \Aut(\bla).\label{proj_Aut}\end{equation}

\begin{lemma}\label{lem:Homhz}
Assume that $\lz(\la)\subset \la^{(1)}$. Under previous notation, we have an exact sequence of groups
\[\begin{array}{r c l}0\longrightarrow\Hom(\la/\la^{(1)},\lz)&\longrightarrow&\Aut(\la)\stackrel{R}{\longrightarrow}\Aut(\bla)\\
u&\longmapsto&\bar u
\end{array}
\]
where $\Hom(\la/\la^{(1)},\lz)$ is seen as the additive vector group.
\end{lemma}
We denote \begin{equation}U:=\{\bar u\;|\;u\in \Hom(I\lb/I\lb^{(1)},\lz(I\lb))\}.\label{def_U}\end{equation} 
This lemma, together with Lemma~\ref{lem:structure_IB}, implies the following results
\begin{coro}\label{coro:U_barIb}
\begin{enumerate}
\item $(U,\circ)$ is a normal subgroup of $\Aut(I\lb)$ of dimension $(\dim \lh)^2$
\item  $R(\Aut(I\lb))=\Aut(I\lb)/U\subset \Aut(\oIlb)$. 
\end{enumerate}
\end{coro}
We will see in Lemma\ref{lem:Autgen} that the last inclusion is actually an equality (\emph{i.e.} the sequence of Lemma \ref{lem:Homhz} is a short exact sequence for $\la=I\lb$)
\begin{proof}[Proof of Lemma~\ref{lem:Homhz}]
We have 
\[(\bar u\circ\bar v)(X)=(X+v(X))+u(X+v(X))=X+u(X)+v(X)=\overline{u+v}(X)\]
where the middle equality is due to $v(X)\in \lz\subset \la^{(1)}\subset \Ker(u)$.
So the map $u\mapsto \bar u$ is
semi-group homomorphism from $(\Hom(\la/\la^{(1)},\lz),+)$ to
$(\End(\la),\circ)$. Since $(\Hom(\la/\la^{(1)}],\lz),+)$ is actually
a group, its image is contained in $\Aut(\la)$. 

It is clear that the map $u\mapsto \bar u$ is injective and, since $u$ takes values in $\lz$, that $R(\bar u)=Id_{\bla}$. In order to prove exactness of the sequence at $\Aut(\la)$, there remains to prove the implication \[\forall \theta\in \Aut(\la), \, R(\theta)=Id_{\bla}\Rightarrow ((\theta-Id)(\la)\subset \lz)\textrm{ and }  ((\theta-Id)_{|\la^{(1)}}=0)\]
The first property is immediate. 
The second one follows from the fact that, for such a $\theta$, we have $\theta([X,Y])\in [X+\lz,Y+\lz]=[X,Y]$. 
\end{proof}

\subsection{The loop subgroup}

\begin{lemma}
\label{lem:Cstar}
  The following map is an injective group homomorphism
$$
\begin{array}{ccl}
  \CC^*&\longto&\Aut(I\lb)\\
	\tau&\longmapsto&\left(
                            \begin{array}{cccll}\delta_\tau\,: &
                              I\lb&\longto&I\lb\\
&x&\longmapsto&x&{\rm if\ }x\in\lb\\
&ty&\longmapsto&\tau ty&{\rm if\ }y\in\lb^{-}\\
                            \end{array}\right ).

\end{array}
$$  
\end{lemma}
We denote by $D\subset \Aut(I\lb)$ the image of this map

\begin{proof}{}
It is a straightforward check on $\lb\ltimes t\lb^{-}$ that the $\delta_{\tau}$ are automorphisms of $I\lb$.
\end{proof}

\begin{remark}
The map
$\delta_\tau$ corresponds to the variable changing $t\mapsto \tau t$ in the $\CC[t]$-Lie algebra $\lbt/t\lnt$. 
Moreover, the Lie algebra of $D$ acts on $I\lb$ like $\CC d$ where $d$ is the derivation involved in the definition of $\lg^{KM}$.
\end{remark}

\subsection{Automorphisms stabilizing the Cartan subalgebra}
\label{sec:AutDtilde}

For any $\alpha\in\Delta(I\lb)$, fix generators $e_{\alpha}$ of $\lgt$, $\alpha\in \tilde \Delta$ giving rise to elements $X_\alpha\in I\lb_{\alpha}$ in the
corresponding root space $(I\lb)_\alpha$. 
 Set
$$
\Gamma:=\left\{\theta\in \Aut(I\lb)\,
\left|
  \begin{array}{l}

\theta(\lh)\subset\lh\\
\theta(\{X_	\alpha\,:\,\alpha\in \Delta(I\lb)\})=\{X_	\alpha\,:\,\alpha\in\Delta(I\lb)\}
  \end{array}
\right .
\right\}.
$$

Note that, since $\lc$ is the sum of $\lh$ with $\lz(I\lb)$ and since the center is characteristic, the elements of $\Gamma$ also stabilze $\lc$. 

\begin{prop}
\label{prop:AutDtilde}
  The group $\Gamma$ is isomorphic to the automorphism 
group of the affine Dynkin diagram of $\lg$. 
\end{prop}

\begin{proof}
By construction, $\Gamma$ induces an action on $\Delta(I\lb)$. 
By Lemma~\ref{Ib2}~(2), we have for $g\in \Gamma$
and $\alpha,\beta\in \Delta(I\lb)$:  
$$
\begin{array}{l@{\;}l@{\;}l}
  a_{\alpha,\beta}&=-\max\{n|(\ad X_{\alpha})^n(X_{\beta})\neq 0\}\\
&=-\max\{n|g((\ad X_{\alpha})^n(X_{\beta}))\neq 0\}\\
&=-\max\{n|(\ad X_{g(\alpha)})^n(X_{g(\beta)})\neq 0\}&=a_{g(\alpha),g(\beta)}.
\end{array}
$$ 
Hence $g$ actually induces an automorphism of the extended Dynkin 
diagram\footnote{If $\lg$ is $\mathfrak{sl}_2$, Lemma~\ref{Ib2}~(2)
  does not apply.
 However, any permutation of $\tilde \Delta$ is an automorphism of the extended Dynkin diagram in this case.} and
we thus obtain a group homomorphism
\[  \Theta\,:\Gamma\rightarrow\Aut(\Dt_\lg).
\]

We claim that $\Theta$ is surjective. 
Indeed, fix an automorphism $\theta$ of the group $\Dt_\lg$. As was mentioned in
Section~\ref{sec:affLA}, there exists $\tilde \theta \in \Aut(\lgt)$
which stabilizes both $\lh$ and $\lbt$ and which permutes the generators $\{e_\alpha: \alpha\in \tilde \Delta\}$ and thus $\tilde\Delta\stackrel{\varphi}{\cong} \Delta(I\lb)$ as $\theta$ does. By Lemma \ref{Ct-lin}, 
$\tilde\theta$ stabilizes $t\lnt$, 
so induces the desired element of $\Aut(\lbt/t\lnt)$.

We now prove that $\Theta$ is injective. 
Let $\theta$ in its kernel. 
By the definition of the group $\Gamma$, $\theta$ stabilizes
$\lh$. 
Since the restrictions of the elements of $\Delta(I\lb)$ span
$\lh^*$, the restriction of $\theta$ to $\lh$ has to
be the identity. 
In particular, $\theta$ acts trivially on $\Phi(I\lb)$ and stabilizes
each root space $(I\lb)_\alpha$ for $\alpha\in \Phi(I\lb)$.
But $\theta$ stabilizes the set $\{X_
\alpha\,:\,\alpha\in\Delta(I\lb)\}$. 
Hence $\theta$ acts trivially on each $\lgt_\alpha$ for $\alpha\in
\Delta(I\lb)$.
Since $\lnt$ is generated by the $(\lgt_\alpha)_{\alpha\in \Delta(I\lb)}$,  
the restriction of $\theta$ to $\lnt/t\lnt$ is the
identity map. 
Finally, $\theta$ is trivial and $\Theta$ is injective.
\end{proof}

\begin{remark}
\begin{enumerate}
\item \cite[Theorem~2]{BNVDV} is the construction of an explicit order $n$ 
automorphism of $\gl^\epsilon_{n+}$. We can also interpret this automorphism in terms of the isomorphism $\gl^{\epsilon}_{n+}\cong \lbt/(t-\epsilon )\lnt$ of Theorem~\ref{th:main}.
Indeed, let $\theta$ be the cyclic automorphism of the extended Dynkin diagram in type $A_{\ell}$ and let $\tilde \theta$ be the automorphism of $\lg$ associated to $\theta$ as in Section~\ref{sec:affLA}. By Lemma~\ref{Ct-lin} and the subsequent remark, $\tilde \theta$ induces an automorphism of $\lbt/(t-\epsilon)\lnt$.  Moreover, it is easily checked that the action on layer $1$ in \cite{BNVDV} is a cyclic permutation of the generators $(e_{\alpha})_{\alpha\in \tilde \Delta}$.
 
\item Consider the trivial vector bundle
  $\underline{\vV}:=\vV\times\bAA^1$ over $\bAA^1=\Spec(\CC[\epsilon])$. 
The Lie bracket $[\ ,\ ]_\epsilon$ endows $\underline{\vV}$ with a
structure of a Lie algebra bundle meaning that $[\ ,\ ]_\epsilon$ can
be seen as a section of the vector bundle $\bigwedge^2 \underline{\vV}^*\otimes \underline{\vV}$
satisfying the Jaccobi identity. Consider the group
$\Aut(\underline{\vV}, [\ ,\ ]_\epsilon)$ consisting of automorphisms
of the vector bundle $\underline{\vV}$ respecting the Lie bracket
pointwise. Let $\theta\in \Aut(\tilde{\mathcal D})$ and assume that the $\tilde\theta\in \Aut(\lgt)$ is $\CC[t]$-linear (\emph{i.e.} $\lambda=1$ in Lemma~\ref{Ct-lin}). Then it is easy to check that $\tilde\theta$ induces an element of $\Aut(\underline{\vV}, [\ ,\ ]_\epsilon)$. In other words, $\theta$ lifts to a $\bAA^1$-family of automorphisms over the $\bAA^1$-family of Lie algebras $\underline{\vV}$. 
\end{enumerate}
\end{remark}

\section{Description of $\Aut(I\lb)$}
\label{sec:structAut}

In this section, we describe the structure of \[
\Aut(I\lb)=\{g\in\GL(I\lb)\,:\, \forall X,Y\in I\lb\qquad
g([X,Y])=[g(X),g(Y)]\} 
\]
in terms of the subgroups $U$, $\Ad(IB)$, $D$ and $\Gamma$ introduced in Section~\ref{sec:subgroups}.

Observe that $\Aut(I\lb)$ is a Zariski closed subgroup of the linear group $\GL(I\lb)$.

\begin{theo}\label{theo:structure_Aut}
We have the following decompositions
\[\Aut(I\lb)=\Gamma\ltimes(D\ltimes(\Ad(IB)\times U)),\] 
\[\Aut(\oIlb)=\Gamma\ltimes(D\ltimes(\Ad(IB)).\] 
In particular, the neutral component is $\Aut(I\lb)^{\circ}=D\ltimes(\Ad(IB)\times U)$ and $\Gamma\cong \Aut(\tilde{\mathcal D}_{\lg})$ can be seen as the component group of $\Aut(I\lb)$.
\end{theo}
The result is a consequence of the lemmas provided below. 
Indeed, by Lemma~\ref{lem:Autgen}, the four subgroups generate
$\Aut(I\lb)$. By Corollary~\ref{coro:U_barIb}(1) and
Lemma~\ref{lem:U_AdIB} below, the subgroup generated by $U$ and $\Ad(IB)$ is a direct product $U\times \Ad(IB)$. Then the structure of $\Aut(I\lb)$ follows from Lemma~\ref{lem:DG}. That of $\Aut(\oIlb)$ follows the same lines, using Corollary~\ref{coro:U_barIb}(2). Note that we have identified $\Gamma$, $\Ad(IB)$ and $D$ with their image under $R$, via Lemma~\ref{lem:Homhz}.

Since $D$, $\Ad(IB)$ and $U$ are connected and $\Gamma$ is discrete,  $\Aut(I\lb)=\bigsqcup_{g\in \Gamma} gD\Ad(IB)U$ is a finite disjoint union of irreducible subsets of the same dimension.  They are thus the irreducible components of $\Aut(IB)$ and the remaining statements of Theorem~\ref{theo:structure_Aut} follow.\\

\begin{lemma}
\label{lem:U_AdIB}
The subgroups
$U$ and $\Ad(IB)$ are normal in $\Aut(I\lb)$. Moreover, $U\cap \Ad(IB)=\{\Id\}$. 
\end{lemma}
\begin{proof}
Recall that $\Ad(IB)$ is generated by the exponentials of $\ad(x)$ with $x\in I\lb$. Then for any $\theta\in \Aut(I\lb)$,  \[\theta \Ad(IB)\theta^{-1}=\theta \exp(I\lb)\theta^{-1}=\exp(\theta(I\lb))=\exp(I\lb)=\Ad(IB).\]

Let $(b,f)\in IB$ and $h\in \lh$. Then $\Ad(b,f)(h)=b\cdot h+t\,b\cdot ([f,h]+\ln)$. Assuming that $\Ad(b,f)=\bar u\in U$, we have $\Ad(b,f)(\lh)\subset \lh+\lz$ so $\Ad(b)(\lh)\subset \lh$, that is $b$ belongs to the normalizer of $\lh$ in $B$, which turns to be $T$. In particular, $b\cdot [f,\lh]\subset \ln^-$ and $Ad(b,f)(\lh)\subset \lh+(\ln+t\ln^-)$. Hence $u=0$ and finally $\Ad(IB)\cap U=\{\Id\}$. 
\end{proof}

\begin{lemma}
\label{lem:Autgen}
We have $\Aut(I\lb)=\Gamma D\Ad(IB)U$ and $\Aut(\oIlb)=\Gamma D\Ad(IB)$.
\end{lemma}

\begin{proof}{}
  Let $\theta\in \Aut(I\lb)$. 
Since the two Cartan subalgebras $\lc$ and $\theta(\lc)$ are
$\Ad$-conjugate (see \cite[\S 3, n$^\circ$ 2, th. 1]{Bourb:Ch78}), 
there exists $\theta_1\in \Ad(IB)\theta$ which stabilizes $\lc$.

Then $\theta_1(\lh)$ is complementary to the center $t\lh=\theta_1(t\lh)$ in $\lc$. Thus, there exists $\theta_2\in U\theta_1$ such that $\theta_2$ stabilizes $\lh$.

Since $\theta_2$ stabilizes $\lc$, it acts on $\Phi(I\lb)$. 
Moreover, $I\lb^{(1)}=[I\lb,I\lb]$ and $I\lb^{(2)}=[I\lb^{(1)},I\lb^{(1)}]$ are
characteristic and stabilized by $\theta_2$.  
So,  Lemma~\ref{Ib2} implies that $\theta_2$ stabilizes
$\Phi(I\lb)\setminus\Delta(I\lb)$ and hence $\Delta(I\lb)$.
Arguing as in the proof of Proposition~\ref{prop:AutDtilde}, we show that the induced permutation is actually an automorphism of the extended Dynkin diagram. Thus there exists $\theta_3\in \Gamma\theta_2$ with the additional property that the induced permutation on $\Delta(I\lb)$ and thus on $\Phi(I\lb)$ are trivial.
Then $\theta_3$ acts on each $(I\lb)_\alpha$ for
$\alpha\in\Delta(I\lb)$.

Since $\Delta$ is a basis of $\lh^*$, one can find $h\in H\subset B\subset IB$ such that
$\Ad(h)\circ\theta_3$ acts trivially on each $(I\lb)_\alpha$ for
$\alpha\in\Delta$. 
Moreover, $D$ acts trivially on these roots spaces and with weight $1$ on $(I\lb)_{\alpha_0}$.
This yields $\theta_4\in D\Ad(H)\Gamma U\Ad(IB)\theta$ which acts trivially on $\lh$ and on each $(I\lb)_{\alpha}$, $\alpha\in \Delta(I\lb)$. 

Recall now that 
$\lnt/t\lnt$ is generated by the spaces
$((I\lb)_\alpha)_{\alpha\in\Delta(I\lb)}$. 
Since $\theta_4$ acts trivially on $\lnt$ and on $\lh$, it has to be trivial. 
As a consequence,
$\theta\in \Ad(IB) U\Gamma\Ad(H) D=\Gamma D\Ad(IB)U$, the last equality following from Lemma~\ref{lem:U_AdIB} and Corollary~\ref{coro:U_barIb}.

Recalling that $\Phi(I\lb)=\Phi(\oIlb)$, the same proof applies for $\oIlb$ instead of $I\lb$, replacing $\lc$ by $\lh$  and skipping step from $\theta_1$ to $\theta_2$.
\end{proof}

\begin{lemma} \label{lem:DG}
The intersections
$D\cap (\Ad(IB)\times U)$ and
$\Gamma\cap (D\ltimes (\Ad(IB)\times U))$ are the trivial group $\{\Id\}$. 
Moreover, $(D\ltimes (\Ad(IB)\times U))$ is normal in $\Aut(I\lb)$.\\
\end{lemma}
\begin{proof}
Let $\tau\in\CC^*$, $b\in B$, $f\in\lg/\ln$ and $u\in
\Hom(I\lb/[I\lb,I\lb],\lz(I\lb))$ such that the associated  elements
$\delta_\tau\in D$, $(b,f)\in IB$ and $\bar u\in U$ (see
Section~\ref{sec:subgroups}) satisfy
$\delta_{\tau}=\Ad(b,f)\circ \bar u$.
 For $x\in \lb$, we have 
\[x=\delta_{\tau}(x)=(\Ad(b,f)\circ \bar u)(x)=\Ad(b,f)(x+u(x))=b\cdot x+ (b\cdot u(x)+tb\cdot([f,x]+\ln)).\] In particular, $b\cdot x=x$ and, whenever $x\in \ln$, $b\cdot [f,x]=0$ in $\lg/\ln$. So $b\in B$ centralizes $\lb$ and $\ad_{\lg}f$ normalizes $\ln$.  As a consequence, $b=1_B$, $f$ is $0$ in $\lg/\lb$ and $u=0$. Thus the only element of $D\cap (\Ad(IB)\times U)$ is the trivial one.

Since $[I\lb,I\lb]$ is characteristic in $I\lb$, we have a natural group morphism $p:\Aut(I\lb)\rightarrow \Aut(I\lb/[I\lb,I\lb])$.
From the description of $[I\lb,I\lb]$ in Lemma~\ref{lem:structure_IB}, it is straightforward that $D$, $\Ad(IB)$ and $U$ are included in $\Ker(p)$ while $p_{|\Gamma}$ is injective. From Lemma~\ref{lem:Autgen}, we then deduce that $D\ltimes (\Ad(IB)\times U)=\Ker(p)$ and the desired properties follow.
\end{proof}


\bibliographystyle{alpha}

\bibliography{barnatan_BR}

\begin{center}
  -\hspace{1em}$\diamondsuit$\hspace{1em}-
\end{center}
\end{document}